\documentclass[11pt,reqno]{amsart}
\input xy
\xyoption{all}
\input epsf
\usepackage[skins]{tcolorbox}

\usepackage{float}

\usepackage{enumitem}
\usepackage {amssymb,color}
\usepackage{verbatim}
\input xy
\xyoption{all}
\input epsf

\newcommand{\red}[1]{{#1}}

\newcommand{\GL}{\mathrm{GL}}

\newcommand{\ol}{\overline}

\newcommand{\field}{\mathbb}

\newcommand{\C}{{\field C}}
\newcommand{\R}{{\field R}}

\newcommand{\Q}{{\field Q}}

\newcommand{\lam}{\lambda}

\newcommand{\SO}{\mathrm{SO}}
\newcommand{\SL}{\mathrm{SL}}

\newcommand{\lra}{\longrightarrow}

\newcommand{\per}{\textrm{per}}

\newcommand{\con}{\textrm{con}}

\newtheorem{prop}{Proposition}[section]
\newtheorem{cor}[prop]{Corollary}
\newtheorem{lemma}[prop]{Lemma}
\newtheorem{theorem}[prop]{Theorem}

\numberwithin{equation}{section}

\newtheorem{corollary}[prop]{Corollary}

\theoremstyle{definition}
\newtheorem{defn}[prop]{Definition}
\newtheorem{remark}[prop]{Remark}
\newtheorem{example}[prop]{Example}

\newcommand{\IB}{A} 

\newcommand{\frb}{\mathfrak{b}}

\newcommand{\frg}{\mathfrak{g}}

\newcommand{\frn}{\mathfrak{n}}

\newcommand{\frp}{\mathfrak{p}}

\newcommand{\frt}{\mathfrak{t}}

\newcommand{\bbC}{\mathbb{C}}

\newcommand{\bbR}{\mathbb{R}}

\newcommand{\bbZ}{\mathbb{Z}}

\newcommand{\caB}{\mathcal{B}}
\newcommand{\caC}{\mathcal{C}}

\newcommand{\caO}{\mathcal{O}}
\newcommand{\caP}{\mathcal{P}}

\newcommand{\caV}{\mathcal{V}}

\newcommand{\caX}{\mathcal{X}}

\newcommand{\beq}{\begin{equation}}
\newcommand{\eeq}{\end{equation}}

\def \Ad  {\mathop{\hbox {Ad}}\nolimits}
\def\ad {\mathop{\hbox {ad}}\nolimits}

\renewcommand{\c}{\mathrm{con}}

\begin{document}
\title[orbits of a symmetric subgroup]{A note on the orbits of a symmetric subgroup in the flag variety}
\author{Leticia Barchini}
\address{Department of Mathematics, Oklahoma State University, Stillwater, OK 74078}
\email{leticia@math.okstate.edu}

\author{Peter E.~Trapa}
\address{Department of Mathematics, University of Utah, Salt Lake City, UT 84112-0090}
\email{peter.trapa@utah.edu}

\maketitle

{\centering\footnotesize \em To Toshiyuki Kobayashi on the occasion of his 60th birthday.\par}

\begin{abstract}
Motivated by relating the representation theory of the split real and $p$-adic forms of a connected reductive
algebraic group $G$, we describe a subset of $2^r$ orbits on the complex flag variety for a certain symmetric subgroup.  (Here $r$ is the semisimple rank
of $G$.)
This set of orbits has the property that, while the closure of individual orbits are generally singular, they are always smooth along other orbits in the set.  This, in turn,
implies consequences for the representation theory of the split real group.
\end{abstract}

\section{introduction}

The main result of this note is a simple statement about the orbits of a certain symmetric subgroup on the flag variety for a complex reductive Lie algebra $\frg$.  As we explain below, it relies on deeper motivation to prove.  Setting aside that motivation for the moment, the result can be stated as follows.  Let $\frb = \frt \oplus \frn$ be a Borel subalgebra of $\frg$, and write $\Pi$ for the corresponding simple roots.  Let $(G, K)$ denote the complex symmetric pair corresponding to the equal-rank quasisplit real form of $G$.    Fix an ordering of $\Pi$.  Then to each subset $S$ of $\Pi$,
we attach an orbit $Q_S$ of $K$ on the flag variety $\caB$ for $\frg$ such that $Q_\emptyset$ is closed;
\begin{equation}
\label{e:intro-dim}
\dim(Q_S) = \dim(Q_\emptyset) + |S|;
\end{equation}
and
\begin{equation}
\label{e:intro-closure}
\text{$Q_{S'} \subset \overline Q_S$ iff $S' \subset S$.}
\end{equation}
\red{
The orbit closure $\overline{Q}_S$ contains many other orbits besides the various $Q_{S'}$ and is generally highly singular.  However, if we further assume $K$ is connected,  
$\overline{Q}_S$ is
always smooth along the orbits $Q_{S'}$.
} 
In particular,
\begin{equation}
\label{e:intro-klv}
p_{Q_{S'},Q_S} = 
\begin{cases} 1 &\text{if $S' \subset S$}\\
0 &  \text{if $S' \not\subset S$},
\end{cases}
\end{equation}
where $p_{Q_{S'},Q_S}$ is the Kazhdan-Lusztig-Vogan polynomial corresponding to the trivial local systems on $Q_{S'}$ and $Q_S$.  The
definition and properties of the orbit $Q_S$ are given in Proposition \ref{p:res}, Corollary \ref{c:closure}, and
Proposition \ref{p:gfinite}.  The assertion about polynomials is proved in Theorem \ref{t:main}.  \red{See Remark \ref{r:conn} for a discussion of the hypothesis that $K$ is connected.}

The proof of the results of the previous paragraph begins to reveal our motivation.  Let $T$ be the complex torus corresponding to $\frt$.  It acts
with finitely many orbits on the span $\frg_{-1}$ of the negative simple root spaces of $\frt$ in $\frg$.  (As the notation indicates, another way to think of this span is as the $-1$-eigenspace of the adjoint action of $\rho^\vee$ in $\frg$, where $\rho^\vee$ is the half-sum of the simple coroots.) The orbits of $T$ on
$\frg_{-1}$ are easy to describe: there is a bijection between subsets of the simple roots $\Pi$ and $T$ orbits which takes a subset $S$ to the $T$ orbit 
through a vector in $\frg_{-1}$with nonzero components in exactly the negative roots spaces determined by $S$.  (See \eqref{e:Os} below.) The closure order on such orbits  corresponds to the inclusion partial order on subsets of $\Pi$, and all orbit closures are smooth.  \red{For each choice of order of $\Pi$}, we define a $T$-equivariant map in Section \ref{s:orbits},
\begin{equation}
\label{e:intro-e}
\epsilon \; : \; \frg_{-1} \longrightarrow \caB,
\end{equation}
prove it has good properties, and deduce \eqref{e:intro-closure} and \eqref{e:intro-klv} from the corresponding statements about $T$ orbits.  \red{We emphasize that the 
map $\epsilon$ and the 
definition of the orbits $Q_S$ that appear above depends on the choice of ordering on $\Pi$; see Examples \ref{e:gl4} and \ref{e:sp4} for explicit examples of this dependence.  But the good geometric relationship summarized in Theorem \ref{t:main} (and its proof) is independent of the ordering considered.}

The map $\epsilon$ in \eqref{e:intro-e} is, in disguise, a relationship between a version of unramified Langlands parameters for a split reductive $p$-adic group and the Langlands parameters (or rather the Adams-Barbasch-Vogan \red{reformulation} of them) of a corresponding split real group.  This is the motivation mentioned above.  \red{Describing such a 
relationship requires recalling large parts of \cite{v:llc}, and we will not attempt to do so here.  We will instead briefly recall how the geometry of $K$ orbit closures on $\caB$ (respectively, the $T$ orbit closures on $\frg_{-1}$) control the multiplicities of  composition factors in certain standard modules for a corresponding real group (respectively, a corresponding $p$-adic group).}

From \red{the perspective of the previous paragraph}, the geometric considerations above should be thought of ``on the dual side'', and they translate into precise
representation theoretic statements.  In more detail, let $G'_\bbR$ be the split real form of the Langlands dual of $G$, and let $I(\rho)$ denote the spherical principal series of $G_\bbR'$ with trivial infinitesimal character.   If we apply the duality of \cite{vogan:IC4} to \eqref{e:intro-klv}, \red{
we conclude that 
to each subset $S \subset \Pi$, there is an irreducible subquotient $J_S$ of $I(\rho)$ that occurs with multiplicity 1.  (For example, $J_\emptyset$ is the trivial representation.) Note that just as $\overline Q_S$ can include many other orbits besides the various $Q_{S'}$ and the singularities along these other orbits can be complicated, there are generally far more composition factors of $I(\rho)$ than the various $J_S$, and these composition factors can occur with complicated multiplicities in $I(\rho)$ described by the Kazhdan-Lusztig-Vogan algorithm. Again, we emphasize that the definition of $Q_S$, and therefore $J_S$, depends on a choice of order of $\Pi$.  If we change the choice of order, we potentially get a different set of composition factors that have multiplicity one.  But, even by varying over all orders, we still see only a generally small subset of all composition factors of $I(\rho$).}
 
 Now, switching to the $p$-adic side, let $F$ be a $p$-adic field and let $G'_F$ denote the $F$-points of the split form of a connected reductive 
algebraic group (dual to $G$) defined over the algebraic closure $\bar F$. Results of Casselman, as explained for example in \cite[Section 4.9]{v:llc}, show that there is an unramified
principal series of $G'_F$ with one-dimensional subquotient whose composition factors are parameterized by the subsets $S$ of $\Pi$, and
each irreducible subquotient occurs with multiplicity one.   \red{In this way, one can view our results as a kind of multiplicity-preserving correspondence (depending on the ordering of $\Pi$) between these subquotients  and a subset of the irreducible subquotients of the spherical principal series of the real group.}

The next natural generalization of these kinds of results is to look beyond the unramified representations of the $p$-adic group to other members of the L-packets that contain them (\cite{lu:cls1,lu:cls2}).  Geometrically, this corresponds to considering other blocks of nontrivial local systems on the orbits $Q_S$ and the corresponding $T$ orbits on $\frg_{-1}$.  Unfortunately, but perhaps not surprisingly, this requires case-by-case considerations.  We carry some of them out in Example \ref{e:classical}.

Results of the kind established here were first obtained for $\mathrm{GL}(n)$ in \cite[Section 2.5]{ct}.  Our present results generalize those of \cite{ct}
to all split groups using different arguments, but only in the restrictive setting of regular integral infinitesimal character.  For classical groups, we will relax this latter hypothesis in a future
paper based on the techniques developed here \red{\cite{bt-singular}}.  Part of our motivation is a view toward
how the unitary duals of a real and $p$-adic group are related: the unitarity algorithm of \cite{altv} and (as of yet unproven) natural analogs for $p$-adic groups rely heavily on 
fine structure of the composition series of standard induced modules; so relating the latter is a natural prerequisite to relating the unitarity algorithms.

Finally, it is a pleasure to thank Professor Toshiyuki Kobayashi for his long and inspiring career, replete with wonderful and deep results, from which we have learned immensely.  

\section{matching of orbits}
\label{s:orbits}
In this section we explain the map $\epsilon$ described in the introduction.  Let $G$ be a complex connected reductive algebraic group with Lie algebra $\frg$.  Fix a Borel subalgebra $\frb = \frt \oplus \frn$, and let $B=TN$ denote the corresponding Borel subgroup of $G$.    Write $\Pi$ for the corresponding set of simple roots, and for each $\alpha \in 
\Pi$, let
\[
\frp_\alpha = \frg_{-\alpha} \oplus \frb,
\]
\red{where} $\frg_{-\alpha}$ denotes the $-\alpha$ root space.   Let $P_\alpha$ denote the corresponding subgroup of $G$.

Let $\lambda$ be a semisimple element of $\frt$.   Let $\frg_{k}$ denote the $k$-eigenspace for $\ad(\lambda)$ acting
on $\frg$,
\[
\frg_k = \{ x \in \frg \; | \; \ad(\lambda)x = kx \}.
\]
Then the centralized in $G$ of $\lambda$, $G(\lambda)$, acts via $\Ad$ with finitely many orbits on each $\frg_k$; \red{see \cite{Vi}}.

In what follows, we will be interested in only the very special case of  $\lambda = \rho^\vee$, one-half the sum of the coroots for $\frt$ in $\frg$.   View $\lambda$, as we may, as an element of $\frt$.  Then all eigenspaces for the action of $\ad(\lambda)$ on $\frg$ are integral, and $G(\lam) = T$. 
We will be concerned only with the $-1$-eigenspace, and it is easy to see that it is spanned by the root spaces for the
negative simple roots,
\[
\frg_{-1} = \bigoplus_{\alpha \in \Pi} \frg_{-\alpha}.
\]
 (No confusion arises in practice by introducing both integer and root subscripts on $\frg$.) The $T$ orbits on $\frg_{-1}$ are very easy to describe.  For each subset $S \subset \Pi$, let
$x_S \in \frg_{-1}$ denote a vector with nonzero components in exactly those $\frg_{-\alpha}$ for $\alpha \in S$.
Then
\begin{equation}
\label{e:Os}
\caO_S := T\cdot x_S,
\end{equation}
consists of {\em all} such vectors with nonzero components in exactly those $\frg_{-\alpha}$ for $\alpha \in S$.  The closure
of $\caO_S$ consists of all vectors with components (possibly zero) in those $\frg_{-\alpha}$ for $\alpha \in S$.
From this it follows, as noted in the introduction, that the closure on the $T$ orbits in $\frg_{-1}$ corresponds to the inclusion partial order on $\Pi$.
\red{This is a special case of the more general framework of Zelevinsky \cite{zel-1,zel-2}, partly based on evidence provided by the results of Casselman
mentioned in the introduction.}

Let $y(\lam) = \exp(i\pi\lam)$.  By the integrality of $\lam  = \rho^\vee$, $y(\lam)^2$ is central in $G$.  Let $\theta$
denote the involution of $G$ obtained by conjugation by $y(\lam)$, and let $K$ be the fixed points of $\theta$.  Since 
$T$ centralizes $\lambda$, $T\subset K$.  (In other words, the symmetric pair $(G,K)$ is equal rank.)  Moreover, 
again by construction, the differential of $\theta$ sends each root vector $x_\alpha \in \frg_\alpha$ to $-x_\alpha$, for 
$\alpha \in \Pi$.   Equivalently, the Borel subalgebra $\frb$ is $\theta$-stable and every simple root is noncompact
imaginary.  (In other words, \red{$(G,K)$ is equal rank and quasisplit.})  Let $\caB \simeq G/B$ denote the variety of Borel algebras in $\frg$.  Then $K$ acts with finitely many orbits on $\caB$, and (since $\frb$ is $\theta$-stable) the orbit $K\cdot \frb$ is closed.  

We seek to define the map $\epsilon$ of \eqref{e:intro-e}.  For this, we fix a choice of ordering of $\Pi$,
\begin{equation}
\label{e:pi-ordered}
\Pi = \{ \alpha_1, \alpha_2, \dots, \alpha_n\}.
\end{equation}
In what follows, the map $\epsilon$ depends on this choice of order, but the result we state below are valid for all choices.

\begin{defn} Let $S = \{ \alpha_{j_1},  \alpha_{j_2}, \ldots,  \alpha_{j_s}\}$ be an (ordered) subset of $\Pi$ consisting of distinct roots.  To simpliy notation, write $P_{k}$ for the parabolic subgroup corresponding to $\alpha_{j_k}$ defined above.
Define
\[ \caX^S = K\underset{K\cap B}{\times}P_{s}\underset{ B}{\times}P_{{s-1}}\underset{ B}{\times}\ldots \underset{ B}{\times}P_{1}\]
the quotient of  $K\times P_{s}\times \ldots \times P_{1}$ by the action
\[ (b_0, b_s, b_{s-1}, \ldots, b_{1}) \cdot  (k_0, y_{s},  y_{{s-1}}, \ldots , y_{1}) =
(k_0 \;b_0, b_0^{-1}\; y_{s} \;b_{s},   b^{-1}_{s}  \;  y_{{s-1}} \;b_{ s-1} , \ldots, b_1^{-1}\; y_{1}).
\] 
(If $S= \emptyset$, this reduces to $\caX^\emptyset = K/(K\cap B)$.)
Let $K$ act on $\caX^S$ via 
\[ k \cdot [k_0, y_{s},  y_{{s-1}}, \ldots , y_{1}] =   [ k \;k_0, y_{s},  y_{{s-1}}, \ldots , y_{1}].\]
\end{defn}

\medskip

\begin{lemma}  
 $\caX^S$ is a smooth variety of dimension $ \text{dim}(K\cdot \frb)  + |S|.$
 If $K$ is connected, $\caX^S$ is irreducible.
\end{lemma}

\begin{proof} This is well-known. See \cite[Section 6]{be} and \cite{ewy}.
\end{proof}
\medskip

\begin{prop}  \label{p:res} 
 For $S \subset \Pi,$ let
 \begin{align*}
 \tau^S : \caX^S & \longrightarrow \caB\\
 [k, y_{s},  y_{{s-1}}, \ldots , y_{1}] & \longrightarrow k \text{exp}(y_{s})\ldots \text{exp}(y_{1}) \cdot \frb.
 \end{align*}
 The map   $\tau^S$ is a well-defined $K$-equivariant map. Moreover, 
 $\tau^S(\caX^S)$ is the closure of a single $K$ orbit which we define to be $Q_S$. 
 
 Let $\pi_{k}$ denote the natural projection from $G/B$ to $G/P_{k}$.  Then, \red{the} closure of $Q_S$
 is given by
  \begin{equation}
  \label{e:Qs}
  \overline{Q}_S =  \left ((\pi_{s}^{-1}\circ\pi_s) \circ \cdots \circ (\pi_{1}^{-1}\circ\pi_1)\right ) (Q_\emptyset),
  \end{equation}
where, by definition,  $Q_\emptyset= K\cdot \frb$.
 
 \end{prop}
 \begin{proof}
This is more or less obvious from the definitions.  See Lemma 6.1 in \cite{v:ic3}, or the exposition
of  \cite[Section 6]{be} and \cite[Section 4]{ewy}.
\end{proof}
\medskip

\begin{cor}
\label{c:closure}
Fix $S,S' \subset \Pi$ and define $K$ orbits $Q_S$ and $Q_{S'}$ as in Proposition \ref{p:res}.  Then $Q_{S'} \subset \overline{Q_S}$ iff  $S' \subset S$ as in
\eqref{e:intro-closure}.
\end{cor}

\begin{example}
\label{e:gl4}
Let $G = \GL(4,\bbC)$.  Label the simple roots so that $\alpha_2$ corresponds to the middle node
of the Dynkin diagam.  (In other words, $\alpha_2$ is the unique simple root not orthogonal to any other
simple roots.)  In this
setting,  $K \simeq \GL(2,\bbC) \times \GL(2,\bbC)$.  The orbits of $K$ on $\caB$
are well-known, and parametrized in terms of certain involutions with signed fixed points called clans; see
\cite{mt,y}.  The closure order is given in \red{Figure 1}.  In the figure, a 
solid edge $Q' \stackrel{i}{\rightarrow} Q$ indicates that $Q$ is dense in $\pi_i^{-1}(\pi_i(Q'))$, and so $Q'$
is in the closure of $Q$.  The dashed edges indicate closure relations not obtained in this way.  
If we fix any of the four orders of the simple roots so that $\alpha_2$ does {\em not} appear last in the order, then the boxed 
orbits are the set of orbits of the form $Q_S$ for $S \subset \Pi$.  However, if we choose one of the two orders
so that $\alpha_2$ appears last, then the set of orbits of the form $Q_S$ are the boxed orbits {\em except}
that the orbit labeled by \red{$1\!+\!-\!1$} does not correspond to $S=\Pi$, and instead the shadow-boxed orbit labeled $1212$
does.  A hand calculation (or using the {\tt atlas} software) implies that the closures of $Q_{1+-1}$ and $Q_{1212}$ are
both smooth along all other boxed orbits; but also reveals that the closure of  $Q_{1+-1}$ is singular along $Q_{++--}$, while the closure of $Q_{1212}$ is singular along $Q_{+--+}$ and $Q_{-++-}$.
\end{example}

\begin{figure}
\label{f:u22-example}
$${\tiny
{
\xymatrixcolsep{2pc} 
\xymatrixrowsep{3pc}
\xymatrix@d
{
&&&&&1221\\
&&&\boxed{1+-1} \ar@{>}[urr]^2 && \text{\tcbox[enhanced,size=fbox,drop shadow southwest,
    sharp corners]{1212}} \ar@{>}[u]^{1,3} && 
{1-+1}\ar@{>}[ull]^2 \\
&\boxed{1+1-} \ar@{>}[urr]^3  \ar@{.>}[urrrr]_<<<<<<<<{} 
&& \boxed{+1-1} \ar@{>}[u]_<<<<1 \ar@{.>}[urr]_<<<<<<{} 
&& \boxed{1122} \ar@{>}[u]^2 \ar@{.>}[urr]_<<<<<<{}  \ar@{.>}[ull]^<<<<<{}  
&& -1+1 \ar@{>}[u]^<<<<<1 \ar@{.>}[ull]^<<<<<<{} 
&& 1-1+ \ar@{>}[ull]^3 \ar@{.>}[ullll]^<<<<<<<<{} \\
\boxed{+11-} \ar@{>}[ur]^1 \ar@{>}[urrr]_<<<<<<<<3 
&& \boxed{11+-} \ar@{>}[ul]_<<<<<2 \ar@{>}[urrr]_<<<<<<3  
&& \boxed{+-11} \ar@{>}[ul]^<<<<<2 \ar@{>}[ur]^1  
&& -+11 \ar@{>}[ur]^<<<<<2 \ar@{>}[ul]^1   
&& {11-+} \ar@{>}[ur]_<<<<<2 \ar@{>}[ulll]_<<<<<<3    
&& {-11+} \ar@{>}[ul]^1 \ar@{>}[ulll]^<<<<<<<<3 \\
++--\ar@{>}[u]^2 
 & & \boxed{+-+-} \ar@{>}[ull]_2 \ar@{>}[u]^1\ar@{>}[urr]_<<<<<<3
 && -++- \ar@{>}[ull]^<<<<<<<1\ar@{>}[urr]_<<<<<<3
 && +--+ \ar@{>}[ull]^<<<<<<3\ar@{>}[urr]_<<<<<<1
 && -+-+ \ar@{>}[ull]^<<<<<<3 \ar@{>}[u]^1\ar@{>}[urr]^2
 && --++\ar@{>}[u]^2 
}
}
}
$$
\caption{$G=\GL(4,\bbC)$; see Example \ref{e:gl4}.}
\end{figure}

\begin{example}
\label{e:sp4}
Let $G = \mathrm{Sp}(4,\bbC)$.  Label the simple roots so that $\alpha$ is short and $\beta$ is long.  In this
case, $K\simeq \GL(2,\bbC)$.  The closure order  of $K$ orbits on $\caB$ is given in Figure \ref{f:sp4-example}.
Once again,  a 
solid edge $Q' \stackrel{i}{\rightarrow} Q$ indicates that $Q$ is dense in $\pi_i^{-1}(\pi_i(Q'))$, and so $Q'$
is in the closure of $Q$.  The dashed edges indicate closure relations not obtained in this way.  If we choose the order,
so that $\beta$ appears before $\alpha$, the boxed nodes correspond to the orbits of the form $Q_S$ for
$S \subset \Pi$.  If we instead choose the order so that $\beta$ is last, the same boxed orbits appear, {\em except}
that $Q_\Pi$ corresponds to the shadow-boxed orbits $Q_{1212}$ (rather than $Q_{1+-1}$\red{)}.  A quick hand calculation (confirmed by {\tt atlas}) shows that, in both cases of the choice of ordering, all orbits have smooth closure.  However, note that
$Q_{1+-1}$ admits a nontrivial irreducible $K$ equivariant local system, while the shadow-boxed orbit $Q_{1212}$ does not.
We return to this distinction in Example \ref{e:classical} below.
\end{example}

\begin{figure}
\label{f:sp4-example}
$$
{\tiny
{\xymatrixcolsep{.75pc}
\xymatrixrowsep{2pc}
\xymatrix{
&&&1221\\
& \boxed{1+-1} \ar[urr]^\beta& & \text{\tcbox[enhanced,size=fbox,drop shadow southwest,
    sharp corners]{1212}}\ar[u]^\alpha  &&
1-+1\ar[ull]_\beta \\
& \boxed{+11-}\ar[u]^\alpha \ar@{.>}[urr]_<<<{}& & 
\boxed{1122}\ar[u]_\beta \ar@{.>}[ull]^<<<{} 
\ar@{.>}[urr]_<<<{} &
& -11+\ar[u]_\alpha\ar@{.>}[ull]^<<<{} \\
++-- \ar[ur]^\beta & & \boxed{+-+-} \ar[ul]_\beta \ar[ur]_\alpha& &
-+-+ \ar[ul]^\alpha \ar[ur]^\beta && --++ \ar[ul]_\betaœ
}
}
}
$$
\caption{$G=\mathrm{Sp}(4,\bbC)$; see Example \ref{e:sp4}.}

\end{figure}

\begin{defn}\label{d:themape}
\red{Fix an ordering of} $\Pi = \{\alpha_1, \alpha_2, \ldots,
\alpha_n\}$  as in \eqref{e:pi-ordered}. Write
$z \in \frg_{-1}$ as $z = \sum_i  z_i$ with $z_i$ in the root space for $- \alpha_i.$
Set
\begin{align}\label{emb}
\iota  :  \frg_{-1} & \longrightarrow \caX^{\Pi}\\ \nonumber
z & \to [1, \text{exp}(z_n), \text{exp}(z_{n-1}), \ldots, \text{exp}(z_1)]; \\ \nonumber
\epsilon :  \frg_{-1} &\longrightarrow  \overline Q_\Pi\\ \nonumber
 z& \to \tau^{\Pi} (\iota (z)) = \exp(z_n) \text{exp}(z_{n-1}) \cdots \text{exp}(z_1)\cdot \frb.
 \end{align}
\end{defn} 
\medskip

\begin{lemma}
\label{l:inj}
Let $\caO_S = T\cdot x_S$ be as in \eqref{e:Os}.  Then
\[
\dim(\epsilon(\caO_S)) = \dim(\caO_S).
\]
\end{lemma}
\begin{proof}
Since $\epsilon$ is $T$-equivariant, $Z_T(x_S) \subset  Z_T(\epsilon(x_S))$.  So the result follows from the
other containment
\begin{equation}
\label{e:supset}
Z_T(x_S) \supset Z_T(\epsilon(x_S)).
\end{equation}
Write $x_S = \sum_i z_i$ with possibly some of $z_i$'s equal to zero.  Because $[\frg_{-1},\frg_{-1}] \subset 
\frg_{-2}$, there is a 
$z \in \bigoplus_{k\leq -2} \frg_k$ so that
\begin{align*}
\epsilon(x_S) &:= \exp(z_n) \text{exp}(z_{n-1}) \cdots \text{exp}(z_1)\cdot \frb \\
&=\exp(z_n+ z_{n-1} + \cdots + z_1 + z )\cdot \frb \\
&=\exp(x_S + z)\cdot  \frb.
\end{align*}
If $t \in T$ centralizes $\epsilon(x_S)$, it thus centralizes $x_S+z$.  Since $T$ preserves the grading of $\frg = \bigoplus_k \frg_k$, if $t\in T$ centralizes $x_S +z$, it must centralize
$x_S$, and so \eqref{e:supset} follows.
\end{proof}
\medskip

\begin{lemma} \label{l:freeaction}
Write $T\bar U$ for the opposite Borel subgroup to $B=TU$.  Then $\bar{U}\cap K$ acts  freely on 
\[
[\bar{U} \cap K] \cdot \epsilon (\frg_{-1}).
\]
Moreover, for all $x \in \frg_{-1}$,
\[
\left( [\bar{U} \cap K] \cdot \epsilon (x) \right ) \cap \epsilon(\frg_{-1}) = \epsilon(x).
\]
\end{lemma}
\begin{proof} Suppose $k \in \bar{U}\cap K$ and
$z = \sum z_i \in \frg_{-1}$ are such that 
\begin{equation}\label{inj}
k \cdot \text{exp}(z_n)\; \text{exp}(z_{n-1})\cdots \text{exp}(z_1) \cdot \frb = \text{exp}(z_n) \; \text{exp}(z_{n-1}) \cdots\text{exp}(z_1) \cdot \frb.
\end{equation}
The stabilizer in $G$ of $\frb$ is $B$ and $\bar{U} \cap B = {1}.$   Thus \eqref{inj} implies
$$k \cdot \text{exp}(z_n)\; \text{exp}(z_{n-1})\; \ldots  \text{exp}(z_1) = \text{exp}(z_n) \; \text{exp}(z_{n-1})
 \ldots  \text{exp}(z_1),$$
from which we conclude that $ k = 1$, verifying the first assertion of the lemma. The second assertion follows in a similar way.
\end{proof}

\begin{prop}\label{p:gfinite}
With notation as in Proposition \ref{p:res}, 
\begin{equation}
\label{e:gfinite}
\text{dim}(Q_S) = \text{dim}(Q_\emptyset) + |S|,
\end{equation}
as in \eqref{e:intro-dim}.
\end{prop}

\begin{proof} 
By Proposition \ref{p:res}, 
\[ \text{dim } Q_S \leq \dim(Q_\emptyset) + |S|.\]
We argue that the converse inequality holds. 
Since $\bar{U}\cap K \cdot \epsilon(\caO_S)$ is contained in $Q_S,$
 Lemma \ref{l:freeaction} implies  
\begin{equation}
\label{e:step1}
\dim(Q_S) \geq \dim(\bar{U}\cap K) + \dim(\epsilon(\caO_S)).
\end{equation}
Because $\frb$ is $\theta$ stable $Q_\emptyset = K/(K \cap B)$, and so 
\[
\dim(Q_\emptyset) = \dim(\bar{U}\cap K).
\]
Together with Lemma \ref{l:inj} and $\dim(\caO_S) = |S|$, \eqref{e:step1} becomes
\[
\dim(Q_S) \geq \dim(Q_\emptyset) + |S|,
\]
as we wished to show.
\end{proof}
\medskip

\begin{corollary}\label{c:opendense} If $K$ is connected, 
$[K\cap \bar{B}] \cdot \epsilon (x)$ is dense in $K \cdot  \epsilon (x)$.
\end{corollary}
\medskip

\medskip

\section{matching of geometric multiplicities}
\label{s:klv}
In this section, we further refine the matching of orbits and closure relations of Proposition
\ref{p:res} and Corollary \ref{c:closure}.  Our goal is to show that the computation of local
intersection cohomology for the closure of the $T$ orbit $\caO_S$ on $\frg_{-1}$ (possibly with nontrivial
coefficients) matches the corresponding
calculation for the closure of the $K$ orbit $Q_S$ on $\caB$.  The main result is Theorem \ref{t:main}.

\begin{remark}
\label{r:klv}
If one is only interested in a statement like \eqref{e:intro-klv}, then one can proceed combinatorially case-by-case using recursion
formulas for KLV polynomials in the interval $\bigcup_{S \subset \Pi}Q_S$ between $Q_\emptyset$ and $Q_\Pi$ in
the closure order on $K$ orbits on $\caB$.  However, our aim is to intrinsically relate the orbits $\caO_S$ and $Q_S$ (as we do below).
\end{remark}

\subsection{Preliminaries.}\label{s:cgp}
Suppose $X$ is a complex algebraic variety on which a complex algebraic group $H$ acts with finitely many orbits.
Let  $\caC(H,X)$ be the category of $H$-equivariant constructible sheaves on $X.$ 
Write 
 $\caP(H,X)$   for  the   category of $H$-equivariant
perverse sheaves on $X.$    

Irreducible objects in both categories are parametrized by
the set $\Xi(H,X)$  consisting of  pairs
$(Q,\caV)$ with $Q$ an orbit of $H$ on $X$ and $\caV$ an irreducible $H$-equivariant local system supported on $Q$.
To each $\gamma \in \Xi(X,H)$, we write $\con(\gamma)$ and $\per(\gamma)$ for the corresponding irreducible
constructible and perverse sheaves.  
By taking Euler characteristics, we identify the
Grothendieck group of the categories $\caP(H,X)$ and  $\caC(H,X).$ In this way, we can consider the change of 
basis matrix,
\begin{equation}
\label{e:gengeomdecomp}
[\text{per}(\gamma)] =  \underset{\psi \in \Xi(H,X) }{\sum } (-1)^{d(\psi)} \;  C_{H,X}(\psi, \gamma) [\c(\psi)];
\end{equation}
here $\psi = (Q_\psi, \caV_\psi)$ and  $d(\psi) = \dim(Q_{\psi}).$
The matrix $\left(C_{H,X}(\psi, \gamma)\right)$ is called the {geometric multiplicity matrix.}

\subsection{Induced bundles}
Suppose a group $H$ acts on a variety $X$ with finitely many orbits.  Suppose $H \subset H'$.  Recall
the induced bundle
\[
H' \times_H X
\]
defined by quotienting $H' \times X$ by $(h'h,x) \sim(h',hx)$ for all $h \in H$.  See \cite[Chapter 7]{abv}, for example.

\label{sec:ib}
\begin{prop}\label{p:ib}
\red{Retain the general setting of Section \ref{s:orbits}; in particular, $K$ may be disconnected.}  Then
 the map
\[
\IB : [K  \cap \bar B] \times_{T} \frg_{-1}  \lra [K \cap \bar B] \cdot \epsilon (\frg_{-1})
\]
defined by
\[
\IB(k, x ) \mapsto k \;\epsilon (x)
\]
is a $K \cap \bar B$ equivariant isomorphism.
\end{prop}
\begin{proof}
Clearly $\IB$ is surjective and $K \cap \bar B$ equivariant.  We prove that $\IB$ is injective. Suppose $\IB(k_1, x_1) = \IB(k_2,x_2) .$ 
Write   $x_k = \underset{i = 1}{\sum}^t z^k_{j_i}$ with $z^k_{j_i} $ in the root space for $\frg_{-\alpha_{j_i}}$.
Then  
\[ 
\epsilon (x_k) = \text{exp } (z^k_{j_t})\ldots  \text{exp } (z^k_{j_1})\cdot  \frb.
\]
Write $k_1 = \bar{n}_1\;  t_1$ with $\bar{n}_1\in [\bar{U}\cap K] $ and $t_1 \in T.$
Similarly, write  $k_2 = \bar{n}_2\;  t_2.$
If $t \in T$, we have,
\[
\epsilon (\text{Ad} (t) x_k) = \text{exp } (\text{Ad} (t) z^k_{j_t})\ldots  \text{exp } (\text{Ad} (t) z^k_{j_1})\cdot  \frb.
\]
Then, the equality
 $\IB(k_1, x_1) = \IB(k_2,x_2), $  becomes
\begin{equation}\label{a}
\bar{n}_1 \; \epsilon(\text{Ad} (t_1) x_1)  =  \bar{n}_2\;  \epsilon(\text{Ad} (t_2 ) x_2).\end{equation}
%
By Lemma \ref{l:freeaction},  $\text{Ad}(t_1) x_1 = \text{Ad }(t_2) x_2 .$

 Moreover,
 \[
 \IB(k_1, x_1)  =  \bar{n}_1 \epsilon (\text{Ad} (t_1) x_1) = \bar{n}_1 \; \epsilon(\text{Ad}(t_2) x_2 )
 \]
 which equals
  \[
  \IB(k_2, x_2) =  \bar{n}_2 \; \epsilon (\text{Ad}(t_2) x_2)
  \]
 by hypothesis.
 In particular, $\bar{n}_1 \; \epsilon(\text{Ad}(t_2) x_2 ) = \bar{n}_2 \; \epsilon (\text{Ad}(t_2) x_2) .$
 By Lemma \ref{l:freeaction} again, $\bar{n}_1 = \bar{n}_2.$ 
 We conclude that
\begin{align*}
(k_2, x_2) &\sim (\bar{n}_2\; t_2,   \text{Ad}(t_2^{-1} )  \text{Ad}(t_1) x_1)  \sim (\bar{n}_2, \text{Ad}(t_1) x_1) \\
&\sim (\bar{n}_1, \text{Ad}(t_1 ) x_1)  \sim (k_1 t_1^{-1},  \text{Ad}(t_1 ) x_1) \sim (k_1, x_1),\end{align*}
 as we wished to show.
\end{proof}

\medskip

\begin{cor}
\label{c:ib-matching}
In the setting of Proposition \ref{p:ib}, there
\red{is a} natural correspondence of
$T$ orbits on $\frg_{-1}$ and $K\cap \bar B$ orbits on $[K\cap \bar B]\cdot\epsilon(\frg_{-1})$,
\[
T\cdot x \mapsto [K \cap \bar B] \cdot \epsilon(x).
\]
Write $A_T(x)$ for the component group of the 
centralizer of $x$ in $T$, and similarly for $A_{K \cap \bar B}(\epsilon(x))$.  Then the natural map
$T \rightarrow [K \cap \bar B]$ induces
an isomorphism
\[
A_T(x) \simeq A_{K\cap \bar B}(\epsilon(x)).
\]
The resulting natural bijection
\begin{equation}
\label{e:Xi-ib-matching}
\Xi(K\cap \bar B, [K\cap \bar B]\cdot\epsilon(\frg_{-1})) \rightarrow \Xi(T,\frg_{-1})  
\end{equation}
implements an identification of the geometric mulitplicity matrices of Section \ref{s:cgp},
\begin{equation}
\label{e:ib-matching}
C_{T,\frg_{-1}} = C_{K\cap \bar B, [K\cap \bar B]\cdot\epsilon(\frg_{-1})}.
\end{equation}
\end{cor}

\begin{proof}
According to \cite[Proposition 7.14]{abv}, 
there is a natural correspondence of $T$ orbits on $\frg_{-1}$ and $K\cap\bar B$ orbits
on the induced bundle $ [K  \cap \bar B] \times_{T} \frg_{-1} $ with the corresponding properties as listed in the 
corollary.  Composing with the isomorphism of Proposition \ref{p:ib} completes the proof.
\end{proof}

\medskip

\subsection{Matching}
\label{ssec:matching}
\red{We now assume $K$ is connected to establish consequences of the density statement of Corollary \ref{c:opendense}.  The reason for the connectedness
hypothesis is to ensure that the $K$ orbits $Q_S$ (and their closures) are irreducible.  See Remark \ref{r:conn} for a further discussion.}

Recall that $K\cdot \epsilon(\frg_{-1}) = \overline Q_\Pi$, by definition, and fix $\gamma \in
\Xi(K,\ol Q_\Pi)$.  \red{Fix $x$ in the support of $\con(\gamma)$, and} let $A_K(\epsilon(x))$ denote the component group of the centralizer of $\epsilon(x)$ in $K$,
an elementary 2-group.  Since $A_K(\epsilon(x))$ is abelian, the irreducible constructible sheaf $\con(\gamma)$
has one-dimensional stalks, and since $K$ is connected, its support is irreducible.  It therefore restricts to an
irreducible constructible sheaf on $[K\cap \bar{B}] \cdot \epsilon (x) \red{\subset [K\cap \bar{B}] \cdot \epsilon (\frg_{-1})}$ of the form $\con(\gamma')$ for some
$\gamma'$ in $\Xi(K\cap \bar B [K\cap \bar{B}] \cdot \epsilon (\red{\frg_{-1}})$.  The assignment $\gamma \mapsto \gamma'$ gives a map
\begin{equation}
\label{e:Xi-match2}
\Xi(K,\ol Q_\Pi) \rightarrow \Xi(K\cap \bar B, [K\cap \bar{B}] \cdot \epsilon \red{(\frg_{-1})}).
\end{equation}
The density of Corollary \ref{c:opendense}
implies that the perverse sheaf $\per(\gamma)$ restricts to $\per(\gamma')$.   Thus, if
$\gamma$ and $\delta$ map to $\gamma'$ and $\delta'$ in \eqref{e:Xi-match2}, the 
corresponding entries of the geometric multiplicity matrices of Section \ref{s:cgp} match,
\begin{equation}
\label{e:cgp-match2}
C_{K,\ol Q_\Pi}(\gamma, \delta) =C_{K\cap \bar B, [K\cap \bar{B}] \cdot \epsilon (x)}(\gamma',\delta').
\end{equation}
 If we compose the maps of \eqref{e:Xi-ib-matching} and \eqref{e:Xi-match2}, we obtain a map of parameters
\begin{equation}
\label{e:Xi-pullback2}
\Phi \; : \; \Xi(K, \overline Q_\Pi) \lra \Xi(T,\frg_{-1}),
\end{equation}
which can be described explicitly as follows.
Consider the natural map
\[
 A_{T}(x) \rightarrow  A_{K}(\epsilon(x)).
\]
Since the right-hand side is abelian, we get a pullback on irreducible representations,
\begin{equation}
\label{e:A-pullback}
\widehat{A_{K}(\epsilon(x))} \lra \widehat{A_T(x))}.
\end{equation}
Because irreducible local systems supported on $K\cdot \epsilon(x)$ are parametrized by irreducible representation
of $A_K(\epsilon(x))$, and similarly for irreducible local systems on $\frg_{-1}$ and $A_{T}(x)$, we obtain the
map $\Phi$ of \eqref{e:Xi-pullback2}.
Unwinding the definitions, we see that if $\psi$ is the trivial local
system on $Q_S$, then $\Phi(\psi)$ is the trivial local system on $\caO_S$.  Thus, the matching of the geometric
multiplicity matrices in \eqref{e:ib-matching} and \eqref{e:cgp-match2} implies our main result:

\begin{theorem}
\label{t:main}
Retain the notation of the previous paragraph, especially the definition of $\Phi$ in \eqref{e:Xi-pullback2}.  Recall
that $K$ is assumed to be connected.
\begin{enumerate}
\item[(a)]
As
in \eqref{e:gengeomdecomp},
we define the geometric change of basis matrix as follows,
\begin{equation}
\label{e:gengeomdecomp1}
[\per(\gamma)] =  \underset{\delta \in \Xi(K,\overline Q_\Pi) }{\sum } (-1)^{d(\gamma)} \;  C_{K,\overline Q_\Pi}(\gamma, \delta) [\con (\delta)].
\end{equation}
Similarly, define
\begin{equation}
\label{e:gengeomdecomp2}
[\per(\phi)] =  \underset{\psi \in \Xi(T,\frg_{-1}) }{\sum } (-1)^{d(\psi)} \;  C_{T,\frg_{-1}}(\phi,\psi ) [\con (\psi)].
\end{equation}
Then,
\[
C_{K,\overline Q_\Pi}(\psi, \gamma) = C_{T,\frg_{-1}}(\Phi(\psi), \Phi(\gamma)).
\]
\medskip

\item[(b)]
In the notation of \eqref{e:Qs}, fix ordered subsets $S',S \subset \Pi$, and suppose $\psi$ is the trivial local 
local system on $Q_{S'}$ and $\gamma$ is the trivial local system on $Q_S$.  Then, in the notation
of \eqref{e:Os}, $\Phi(\psi)$ is the trivial local system on $\caO_{S'}$ and $\Phi(\gamma)$ is the trivial local system
on $\caO_S$.  Because the closures of the orbits $\caO_{S'}$ and $\caO_S$ are smooth, we conclude
\[
C_{K,\overline Q_\Pi}(\psi, \gamma) = C_{T,\frg_{-1}}(\Phi(\psi), \Phi(\gamma)) = \begin{cases} 1 &\text{if $S' \subset S$}\\
0 &  \text{if $S' \not\subset S$};
\end{cases}
\]
cf.~\eqref{e:intro-klv}.
\end{enumerate}
\end{theorem}

\qed

\begin{example}
\label{e:classical}
\red{If $K$ is connected --- which is automatic if $G$ is simply connected, for example, as in the examples below --- we always get the matching of Theorem \ref{t:main}}.   But
it is interesting to ask when $\Phi$ in \eqref{e:Xi-pullback2} is surjective; or, equivalently, when the map in
\eqref{e:A-pullback} is surjective as $x$ ranges over representatives of $T$ orbits on $\frg_{-1}.$  In this case, every Kazhdan-Lusztig polynomial
that appears for $T$ orbits on $\frg_{-1}$ is matched with a Kazhdan-Lusztig-Vogan polynomial for $K$ orbits 
on $\caB$ by Theorem \ref{t:main}(a).  Note that $\epsilon$ (and hence $\Phi$) depends on a choice of ordering of $\Pi$.

For $G = \GL(n,\C)$, \red{$K \simeq \GL(\lceil n/2 \rceil) \times \GL(\lfloor n/2 \rfloor)$ is connected}.  All $A$-groups are trivial, so $\Phi$ is clearly surjective.  When the ordering is the standard Bourbaki
ordering, the map $\Phi$ essentially appears in \cite[Section 2.5]{ct}.  The other orderings give rise to maps $\Phi$
that are different (as already seen in Example \ref{e:gl4}).

For $G=\SL(n,\C)$, \red{$K$ is again connected}, but there is no chance for $\Phi$ to be surjective for $n\geq3$.  When $S=\Pi$, $A_T(x_S) = \bbZ/n\bbZ$, while $A_K(\epsilon(x_S))$ is an elementary 2-group.  To obtain a complete matching of polynomials, one must match various blocks of $\Xi(T,\frg_{-1})$ with $K$ orbits constructed in $\SL(n/d,\C)$ for $d|n$.  

When $G=\mathrm{Sp}(2n,\C)$, $K \simeq \GL(n,\bbC)$ is connected, but the choice of ordering on $\Pi$ is important, as Example \ref{e:sp4} already indicates.  In fact, that
example generalizes quite naturally as follows. Let $\beta \in \Pi$
denote the long simple root, and let $\alpha$ denote the unique short simple root not orthogonal to $\beta$.  If $\beta \in S$, then $A_T(x_S) = \bbZ/2\bbZ$; in all other cases  $A_T(x_S)$ is trivial.  If the order on $\Pi$ is chosen so that
$\beta$ appears before $\alpha$, and $\beta \in S$, then $A_K(\epsilon(x_S)) = \bbZ/2\bbZ$; in all other cases the $A_K(\epsilon(x_S))$ is trivial.  Thus $\Phi$ is surjective in this case.  However, if $\beta$ appears {\em after}
$\alpha$ in the ordering on $\Pi$, and $S$ is any subset containing both $\alpha$ and $\beta$, we have  $A_T(x_S) = \bbZ/2\bbZ$
while $A_K(\epsilon(x_S))$ is trivial; so in this case, the map $\Phi$ is not surjective.

For $\mathrm{Spin}(n,\bbC)$, \red{$K$ is again connected.  The situation is similar, and we simply state the results. Suppose first that  $n$ is even, let $\alpha$ and $\beta$ be the  roots at the end of the ``fork'' of the Dynkin diagram.  That is, $\alpha$ and $\beta$ 
are orthogonal to each other; and there is a unique simple root $\gamma$ that is not orthogonal to either $\alpha$ or $\beta$.  If the order on $\Pi$ is chosen so that $\alpha$ and $\beta$ (in either order) appear before all other simple roots, then the map $\Phi$ is surjective (but can fail to be so for other choices).
 In the case of  $n$  odd,  Example \ref{e:sp4} again shows that the choice of $\Pi$ is important since $\mathrm{Spin}(5) \simeq\mathrm{Sp}(4)$.
 Let $\beta$ denote the unique short simple root, and let $\alpha$ denote the unique long simple root
not orthogonal to $\beta$. 
If the order on $\Pi$ is chosen so that $\beta$ appears first, then $\Phi$ is again surjective (but again can fail to be so for other choices)}.

 \end{example}

\begin{remark}
\label{r:conn}
\red{
Suppose $K$ is disconnected, and let $K_e$ denote its identify component.  One can always repeat the constructions above replacing $K$ by $K_e$ and obtain a version of Theorem \ref{t:main}.  (A disadvantage of this approach is that the $K_e$ may no longer correspond to an algebraic real form of $G$.  Passing from $K$ to $K_e$ can also complicate the surjectivity considerations addressed in Example \ref{e:classical}.) If one instead deals with the disconnected $K$ itself, a complication arises when $Q_{S'} \subset \overline{Q_S}$, $Q_{S'}$ is a single $K_e$ orbit, but 
$Q_S$ is a union of multiple $K_e$ orbits $Q_S^i$, each of which has $Q_{S'}$ in its closure.  The image of $\caO_S$ under $\epsilon$ is contained in a single $K_e$ orbit, say $Q_{S}^1$.  But the other irreducible components $\overline{Q_{S}^i}$ for $i > 1$ also contribute to the intersection homology along $Q_{S'}$.  Thus the multiplicities arising in the $K$ orbits setting can be strictly larger than those appearing in the $T$ orbit setting.  
}

\red{
Many interesting phenomena can be seen in the case of $G= \SO(n)$.   In this case, $K$ consists of the determinant 1 elements in $\mathrm{O}(\lceil n/2 \rceil) \times \mathrm{O}(\lfloor n/2 \rfloor)$ and has two components.  If one chooses an order on $\Pi$ which is {\em not} the  order described in Example \ref{e:classical} ensuring surjectivity, then the phenomenon in the previous paragraph can definitely arise.  Indeed, already for $n=8$, one can choose an ordering that leads to instances of $C_{K,\overline Q_\Pi}(\psi, \gamma) =2$ in contrast to the conclusion
of Theorem \ref{t:main}(2).  However, if one chooses the good order of $\Pi$ described in Example \ref{e:classical}, one can perform a direct analysis of the orbits in question to rule out the phenomenon in the previous paragraph, and deduce that the conclusions of Theorem \ref{t:main} hold.   In this case, $\Phi$ is also surjective.  We omit the details.  
}
\end{remark}

\end{document}